\documentclass[11pt, letter]{amsart}

\usepackage[colorlinks = true,
            linkcolor = blue,
            urlcolor  = blue,
            citecolor = black,
            anchorcolor = blue]{hyperref}
\usepackage{color}
\usepackage[shortalphabetic,initials,nobysame]{amsrefs}
\usepackage{upgreek}
\usepackage{tikz}
\usepackage[applemac]{inputenc} % for Macs
\usetikzlibrary{arrows}
\usepackage{xcolor}
\usepackage[all]{xy}
\usepackage{graphicx}
\usepackage{wrapfig}
\usepackage{yfonts}
\usepackage{graphicx}
\usepackage{amssymb}
\usepackage{epstopdf}
\usepackage{mathrsfs}
\usepackage[mathcal]{eucal}
\usepackage{bm}

\renewcommand{\eprint}[1]{\href{https://arxiv.org/abs/#1}{#1}}

\DeclareMathOperator{\GL}{\mathrm{GL}}

\DeclareMathOperator{\SL}{\mathrm{SL}}

\DeclareMathOperator{\Hom}{Hom}

\newcommand{\ev}{\mathrm{ev}}
\newcommand{\odd}{\mathrm{odd}}

\newcommand{\End}{\mathrm{End}}

\newcommand{\CC}{\mathbb{C}}
\newcommand{\ZZ}{\mathbb{Z}}

\newcommand{\Lzero}{\Lambda_0}
\newcommand{\Lone}{\Lambda_1}

\newcommand{\sdet}{\text{sdet}}
\newcommand{\str}{\text{str}}
\newcommand{\Pic}{\text{Pic}}

\BibSpec{article}{%
    +{}  {\PrintAuthors}                {author}
    +{,} { \textit}                     {title}
     +{,} {}                            {note}
    +{.} { }                            {part}
    +{:} { \textit}                     {subtitle}
    +{,} { \PrintContributions}         {contribution}
    +{.} { \PrintPartials}              {partial}
    +{,} { }                            {journal}
    +{}  { \textbf}                     {volume}
    +{}  { \PrintDate}                {date}
    +{,} { \issuetext}                  {number}
    +{,} { \eprintpages}                {pages}
    +{,} { }                            {status}
    +{,} { \DOI}                   {doi}
    +{,} { \eprint}        {eprint}
      +{,} {\publisher}                {publisher}
    +{,} { \address}                   {address}
    +{}  { \parenthesize}               {language}
    +{}  { \PrintTranslation}           {translation}
    +{;} { \PrintReprint}               {reprint}
    +{.} {}                             {transition}
    +{}  {\SentenceSpace \PrintReviews} {review}
}

\newtheorem*{theorem*}{Theorem}

\newtheorem{Thm}{Theorem}[section]

\newtheorem{Prop}[Thm]{Proposition}
\newtheorem{Cor}[Thm]{Corollary}

\theoremstyle{definition}
\newtheorem{Def}[Thm]{Definition}
\theoremstyle{remark}
\newtheorem{Rem}[Thm]{Remark}

\newtheoremstyle{named}{}{}{\itshape}{}{\bfseries}{.}{.5em}{#1 #3}
\theoremstyle{named}

\def\g{\mathfrak{g}}
\def\Frenkel:2013uda{\mathfrak{h}}

\def\bo{\textbf{o}}

\def\=>{\Longrightarrow}

\def\to{\longrightarrow}

\def\o+{\oplus}
\def\bo+{\bigoplus}

\def\<{\langle}
\def\>{\rangle}
\def\({\left(}
\def\){\right)}

\def\^{\wedge}
\def\+{\dagger}

\def\dd[#1,#2]{\frac{d#1}{d#2}}
\def\del[#1,#2]{\frac{\partial #1}{\partial #2}}
\def\over[#1]{\overline{#1}}
\def\vec[#1]{\overrightarrow{#1}}

\makeatletter
\def\mr@ignsp#1 {\ifx\:#1\@empty\else #1\expandafter\mr@ignsp\fi}%
\newcommand{\multiref}[1]{\begingroup%\let\protect\string%
\xdef\mr@no@sparg{\expandafter\mr@ignsp#1 \: }%
\def\mr@comma{}%
\@for\mr@refs:=\mr@no@sparg\do{\mr@comma\def\mr@comma{,}\ref{\mr@refs}}%
\endgroup}
\makeatother

\newcommand{\hypref}[2]{\ifx\href\asklFrenkel:2013udaas #2\else\href{#1}{#2}\fi}

\tikzset{->-/.style={decoration={
  markings,
  mark=at position .5 with {\arrow{latex}}},postaction={decorate}}}
\tikzset{
    >=latex
    }

\newcommand{\nc}{\newcommand}
\nc{\on}{\operatorname}
\nc{\la}{\lambda}
\nc{\wh}{\widehat}
\nc{\ghat}{\wh\g}
\nc{\mb}{\mathbf}

\begin{document}
\thispagestyle{empty}

\title{On $GL(1|1)$ Higgs bundles}

\thispagestyle{empty}

\author[A.M. Zeitlin]{Anton M. Zeitlin}
\address{\hspace{-0.42cm}School of Mathematics,\newline
Georgia Institute of Technology,\newline
 686 Cherry Street, \newline
Atlanta, GA 30332-0160, USA
\newline
Email: \href{mailto:zeitlin@gatech.edu}{zeitlin@gatech.edu}\newline
 URL: \href{https://zeitlin.math.gatech.edu}{https://zeitlin.math.gatech.edu}}

%\numberwithin{equation}{section}

\maketitle
\setcounter{tocdepth}{1}

\begin{abstract}

We investigate the moduli space of holomorphic $\GL(1|1)$ Higgs bundles over a compact Riemann surface. The supergroup $\GL(1|1)$, the simplest non-trivial example beyond abelian cases, provides an ideal setting for developing supergeometric analogues of classical results in Higgs bundle theory. We derive an explicit description of the moduli space and we study the analogue of the Narasimhan-Seshadri theorem as well as the nonabelian Hodge correspondence. Furthermore, we formulate and solve the corresponding Hitchin equations, demonstrating their compatibility with fermionic contributions. As a highlight, we discuss the related Hitchin system on $\mathbb{P}^1$ and its integrability.

\end{abstract}

\tableofcontents

\section{Introduction}

Higgs bundles, first introduced by Nigel Hitchin in his seminal 1987 paper \cite{hitchinsd}, have evolved into a fundamental concept at the intersection of modern algebraic geometry, differential geometry, and mathematical physics. At its core, a Higgs bundle over a compact Riemann surface $C$ consists of a pair 
$(E, \Phi)$, where $E$ is a holomorphic vector bundle, and $\Phi$ is a holomorphic section of ${\rm End}(E) \otimes K_C$, where $K_C$ 
is the canonical bundle of $C$. These structures originated from Hitchin's investigation into self-dual Yang-Mills equations on four-dimensional manifolds, where a process of dimensional reduction 
 -- projecting the equations down to two dimensions -- yields the celebrated Hitchin equations. These equations describe harmonic metrics on bundles and have far-reaching implications for understanding gauge theories in physics.
The moduli space of Higgs bundles, 
which parametrizes solutions to Hitchin's equations up to gauge equivalence, reveals intricate geometric properties. It is equipped with a hyper-K\"ahler structure, a special type of complex symplectic geometry that admits three compatible complex structures, leading to rich metrics and symmetries. Furthermore, this moduli space leads to integrable systems \cite{hitchinint} through the Hitchin fibration, a map that projects the space onto a base consisting of spectral data, often resulting in algebraic completely integrable Hamiltonian systems (see, e.g., \cite{EtingofLiu}). These features make Higgs bundles a powerful tool for studying dynamics in both mathematics and physics, such as in the context of Langlands duality \cite{DP_Langlands_inv}, mirror symmetry \cite{hauselICM} or Seiberg-Witten theory (see, e.g., \cite{Nekrasov:2011bc}).

Beyond their geometric richness, Higgs bundles serve as a crucial link connecting algebraic geometry, representation theory, and differential geometry via the nonabelian Hodge correspondence. This profound theorem, built upon contributions of Donaldson \cite{Donaldsontw}, Hitchin \cite{hitchinsd}, 
Simpson \cite{Simpsonnabh}, \cite{Simpsonproof} and Corlette \cite{Corlette1}, \cite{Corlettenabh} establishes a real-analytic isomorphism between the moduli space of stable Higgs bundles and the moduli space of flat connections on the same Riemann surface. In more detail, for a complex semisimple Lie group $G$, the correspondence associates polystable $G$-Higgs bundles, those that admit no proper Higgs-invariant subbundles of lower degree--with completely reducible representations of the fundamental group $\pi_1(C)$ into $G$. This connection has deep consequences, including applications to character varieties, harmonic maps between manifolds, and even topological invariants of surfaces.
A key pillar underlying this framework is the Narasimhan-Seshadri theorem \cite{NarSesh}: this theorem states that a holomorphic vector bundle over a compact Riemann surface is stable, meaning it has no destabilizing subbundles,  if and only if it originates from an irreducible unitary representation of the fundamental group into the unitary group. Originally inspired by problems in algebraic geometry concerning the classification of bundles, the theorem provided a bridge to flat unitary connections, interpreting stability in terms of metric properties. Donaldson's groundbreaking 1983 proof \cite{DonaldsonNS}, which employed gauge-theoretic techniques from Yang-Mills theory, further intertwined it with differential geometry. Since then, the result has been extended to higher-dimensional varieties, non-compact surfaces, and parabolic bundles, influencing areas like moduli theory and geometric invariant theory.

Although the theory of Higgs bundles is mature 
for classical Lie groups \cite{Swoboda},\cite{DP_Langlands}, its generalization to supergroups \cite{berezin}, blending commuting (bosonic) and anticommuting (fermionic) variables, remains a completely untapped frontier. Supergroups naturally emerge in supersymmetric physical theories, including super-Yang-Mills gauge theories, string theory, and models of quantum gravity, where they account for symmetries between bosons and fermions. The supergroup $\GL(1|1)$, the general linear supergroup of superdimension  $(2|2)$ over a Grassmann algebra, stands out as the simplest non-trivial supergroup beyond purely abelian cases. It comprises invertible supermatrices of the block form
$\begin{pmatrix}
a & \beta \\
\gamma & d
\end{pmatrix},$
where $a$ and $d$ are even (bosonic) elements from a Grassmann algebra $\Lambda$, $\beta$ and $\gamma$ are odd (fermionic) elements, and the Berezinian (or superdeterminant), defined as ${\rm sdet} = \frac{a}{d} \left(1 - \frac{\beta \gamma}{d a}\right)$, is invertible. This supergroup has been explored in contexts like quantum group deformations, noncommutative differential calculi, and supergeometry, but applying it to Higgs bundle theory promises fresh perspectives on superanalogues of classical geometric results, such as stability criteria and moduli constructions.

The impetus for investigating $\GL(1|1)$ Higgs bundles arises from both intrinsic mathematical interest and tangible physical relevance. In theoretical physics, supersymmetric gauge theories frequently feature supergroups to model phenomena like supersymmetry breaking or dualities, and Higgs bundles offer a geometric lens for analyzing their vacuum moduli spaces-the spaces of low-energy configurations. Extending Higgs bundle theory to supergroups like $\GL(1|1)$ may illuminate integrable structures within supersymmetric models, including those in the AdS/CFT correspondence or superconformal field theories that preserve additional symmetries. On the mathematical side, $\GL(1|1)$ provides a computationally accessible testing ground for supergroup analogues of cornerstone theorems. For instance, this paper derives a supergeometric version of the Narasimhan-Seshadri theorem, parametrizing stable $\GL(1|1)$ bundles in terms of flat superconnections, and explores a nonabelian Hodge-type correspondence for supergroups. Moreover, it paves the way for generalizations to higher-rank supergroups, such as $\GL(m|n)$ or orthosymplectic supergroups ${\rm OSP}(m|2n)$, potentially unlocking new classes of  integrable systems and super-character varieties.

 We note here, however, that a significant part of work has been done for the study of the character varieties for the real Lie supergroups from a point of view of super-hyperbolic geometry, namely the corresponding (higher) super-Teichm\"uller spaces of low rank, in particular $OSP(1|2)$ \cite{penzeit}, \cite{mcshane}, \cite{IPZramond},  and $OSP(2|2)$ \cite{ipz}, which are related to the study of super-Riemann surfaces and complex $(1|1)$ supermanifolds \cite{schwarzold}, \cite{superwitten},\cite{schwarz}. 
We also note recent progress on opers in super setting \cite{Zeitlin:2013iya}, \cite{superopers_2025}, which provide an important slice in Higgs bundles supermoduli and related Hitchin integrable system.

In this work, we specifically examine the moduli space of holomorphic $GL(1|1)$ bundles equipped with a Higgs field, 
focusing on this minimal supergroup case to derive explicit results. 
We establish that the moduli space of $SL(1|1)$ bundles 
(the special linear subsupergroup where the Berezinian is 1) is 
isomorphic to ${Pic}_{\Lambda}(C) \times H^1(C, \mathcal{O}_C) 
\otimes \Lambda_1 \times H^1(C, \mathcal{O}_C) \otimes \Lambda_1 $, where 
${Pic}_{\Lambda}(C)$ is the Picard group of line bundles over 
$C$ with $\Lambda$-valued transition functions, and $\Lambda_0, \Lambda_1$ denote 
the even and odd parts of the Grassmann algebra.
 For the full $\GL(1|1)$, this is extended by incorporating a 
 line bundle component related to the Berezinian. 
 We also analyze the Higgs field as an even section of the endomorphism bundle twisted by the canonical bundle. We show the bilinear constraints on fermionic parameters, which carve out the moduli space of $GL(1|1)$ Higgs bundles. Then we study Hitchin equations adapted to the super case and demonstrate their solvability. Furthermore, we parametrize flat $\SL(1|1)$-connections using previous results in  \cite{bourque}, establishing a superanalogue of the nonabelian Hodge correspondence that relates Higgs bundles to flat connections. 
 Finally, we investigate the associated Hitchin integrable system on the projective line $\mathbb{P}^1$, 
 revealing connections to classical integrable models like the Garnier system and the quantum spin chain system -- Gaudin model, thus highlighting the 
  integrability preserved in the supergeometric setting.

The paper is organized as follows: Section 2 sets forth essential notation and conventions for superbundles over Grassmann algebras. Section 3 presents the Gaussian decomposition of the $\GL(1|1)$ supergroup, providing explicit parametrizations for group elements and their products. Section 4 explores transition functions for $\GL(1|1)$ bundles, deriving cocycle conditions and characterizing the moduli space. Section 5 delves into the properties of the Higgs field, including its action and invariants. Section 6 formulates and solves the Hitchin equations for $\SL(1|1)$ bundles, incorporating fermionic contributions, while Section 7 offers a parametrization of flat $\SL(1|1)$-connections, linking them to representations of the fundamental group. Section 8 examines the Hitchin system on $\mathbb{P}^1$, demonstrating its integrability and relations to the Garnier and Gaudin models \cite{KangLu_Gaudin} for $\mathfrak{gl}(1|1)$ Lie superalgebra.

\subsection*{Acknowledgements} The author is grateful to Zhiyang Jin for comments on the manuscript. The author is partially supported by the NSF grant DMS-2526435 (formerly DMS-2203823).

\section{Notation and conventions}
Let $C$ be a compact Riemann surface. The aim of this work is to develop the theory of Higgs bundles for the complex supergroup $\GL(1|1)$ over $C$ within the category of superspaces over a Grassmann algebra. Throughout, we fix a Grassmann algebra
\[
\Lambda = \CC[\theta_1,\theta_2,\dots,\theta_N]/( \theta_i\theta_j + \theta_j\theta_i = 0,\ \theta_i^2=0 ),
\]
where $N \gg 0$ is taken sufficiently large so that all constructions become independent of $N$. All geometric objects (vector spaces, manifolds, sheaves, etc.) will be $\ZZ_2$-graded $\Lambda$-modules unless stated otherwise. For a $\ZZ_2$-graded $\Lambda$-module $V = V_0 \oplus V_1$ we denote the even and odd parts by $V_\ev$ and $V_\odd$, respectively. If $V_0$, $V_1$ are free $n$- and $m$-dimensional $\Lambda$-modules, then we denote the resulting module $\Lambda^{m|n}$. We will also use the concept of the $(n|m)$-dimensional superspace $\mathbb{C}^{m|n}_{\Lambda}$, when discussing coordinate charts for supermanifolds: this is a collection $m|n$-tuples $(z_1, \dots, z_n, \eta_1, \dots, \eta_m)$, 
where $z_i\in \Lambda_0$,  $\eta_i\in \Lambda_1$ as a $\Lambda^{\times}$-module.

\begin{Def}
A \emph{super vector bundle} 
(or \emph{superbundle}) $\mathcal{E}$ on $C$ of rank $(p|q)$ is a locally free sheaf of $\mathcal{O}_C \otimes \Lambda$-modules of rank $p$ in degree $0$ and 
rank $q$ in degree $1$. Locally on an open set $U \subset C$ on which $\mathcal{E}$ is trivial, we have
\[
\mathcal{E}|_U \cong (\mathcal{O}_U \otimes \Lambda)^p \oplus \Pi (\mathcal{O}_U \otimes \Lambda)^q,
\]
where $\Pi$ denotes the parity-change functor.
\end{Def}

Equivalently, a super vector bundle $\mathcal{E}$ is \emph{holomorphic} if it is equipped with an integrable Cauchy--Riemann operator
\[
\bar{\partial}_\mathcal{E} : \mathcal{E} \to \mathcal{E} \otimes \overline{K}_C
\]
that is satisfying the graded Leibniz rule and satisfies $\bar{\partial}_\mathcal{E}^2 = 0$. Equivalently, the transition functions $\{g_{ij}\} \in \GL(p|q;\mathcal{O}_U(U) \otimes \Lambda)$ are holomorphic and even.

%The underlying purely even classical vector bundle is $\mathcal{E}_\ev = (\mathcal{E} / \mathcal{E})|_\theta_i=0$, which is an ordinary holomorphic vector bundle of rank $p+q$.

The supergroup $\GL(1|1) = \GL(1|1; \Lambda)$ consists of even invertible $(1|1) \times (1|1)$ matrices over $\Lambda$. In block form,
\[
\begin{pmatrix}
a & \beta \\
\gamma & d
\end{pmatrix},
\qquad a,d \in \Lambda_0,\ \beta, \gamma \in \Lambda_1,
\]
is an element of $\GL(1|1)$ if and only if its Berezinian (superdeterminant)
\[
\sdet\begin{pmatrix} a & \beta \\ \gamma & d \end{pmatrix}
= \frac{a}{d} \left( 1 - \frac{\beta\gamma}{da} \right) \in \Lambda^\times
\]
is invertible. This is equivalent to $a \neq 0$ and $d \neq 0$ in the body of $\Lambda$.

\begin{Def}
A \emph{holomorphic principal $\GL(1|1)$-superbundle} on $C$ is a holomorphic superbundle $\mathcal{P} \to C$ of relative dimension $(2|2)$ equipped with a free transitive right action of $\GL(1|1)$ that is holomorphic and grading-preserving, and which is locally trivial in the super sense.
\end{Def}

The associated vector superbundle $\mathcal{E} = \mathcal{P} \times_{\GL(1|1)} \Lambda^{1|1}$ is a holomorphic super vector bundle of rank $(1|1)$.

Now let $\mathcal{E}$ be a holomorphic super vector bundle of rank $(1|1)$ on $C$.

\begin{Def}
We define a \emph{Higgs field} on $\mathcal{E}$ as an even global section
\[
\Phi \in H^0(C,\End(\mathcal{E}) \otimes K_C).
\]
\end{Def}

\begin{Def}
A $\GL(1|1)$ \emph{Higgs bundle} on $C$ is a pair $(\mathcal{E},\Phi)$, where $\mathcal{E}$ is a holomorphic super vector bundle of rank $(1|1)$ (equivalently arising from a principal $\GL(1|1)$-superbundle) and $\Phi$ is a Higgs field on $\mathcal{E}$.
\end{Def}

%The central object of study in the subsequent chapters will be the moduli stack (and, where it exists, the coarse moduli space) of semistable $\GL(1|1)$ Higgs superbundles of fixed topological type. Despite the simplicity of the nilpotency condition, the geometry of these moduli spaces turns out to be surprisingly rich and intimately related to spin structures, quadratic differentials, and fermionic extensions of classical bundles.

\section{The decomposition of $GL(1|1)$ supergroup}

The supergroup $\GL(1|1)$ consists of 
$(1|1)\times (1|1)$-dimensional space of supermatrices:

\[
g = \begin{pmatrix}
a & \beta \\
\gamma & d
\end{pmatrix}, 
\quad a, d \in \Lzero, \quad \beta, \gamma \in \Lone,
\]

with Berezinian:

\begin{equation}
\label{eq:sdet}
\sdet(g) = \frac{a}{d} \left( 1 - \frac{\beta\gamma}{da} \right) \neq 0.
\end{equation}

This group contains a natural subgroup $\SL(1|1)$, such that $\sdet(g)=1$.

The group $\SL(1|1)$ admits a 
Gaussian decomposition, where any element $g \in \SL(1|1)$ can be expressed as $g(h,\alpha, \beta) = N_- H N_+$, with:
\begin{itemize}
    \item $N_- = \begin{pmatrix} 1 & 0 \\ \alpha & 1 \end{pmatrix}$, a lower-triangular matrix, $\alpha \in \Lone$,
    \item $H = \begin{pmatrix} e^h \left(1 - \frac{\alpha \beta}{2}\right) & 0 \\ 0 & e^h \left(1 - \frac{\alpha \beta}{2}\right) \end{pmatrix}$, a diagonal matrix, $h \in \Lzero$,
    \item $N_+ = \begin{pmatrix} 1 & \beta \\ 0 & 1 \end{pmatrix}$, an upper-triangular matrix, $\beta \in \Lone$.
\end{itemize}
This decomposition ensures $\sdet(g) = 1$, as required for $\SL(1|1)$. Computing the product $N_- H N_+$:

%\[
%H N_+ = \begin{pmatrix}
%e^h \left(1 - \frac{\alpha \beta}{2}\right) & e^h \left(1 - \frac{\alpha \beta}{2}\right) \beta \\
%0 & e^h \left(1 - \frac{\alpha \beta}{2}\right)
%\end{pmatrix},
%\]

\[
N_- H N_+ = \begin{pmatrix}
1 & 0 \\
\alpha & 1
\end{pmatrix}
\begin{pmatrix}
e^h \left(1 - \frac{\alpha \beta}{2}\right) & e^h \left(1 - \frac{\alpha \beta}{2}\right) \beta \\
0 & e^h \left(1 - \frac{\alpha \beta}{2}\right)
\end{pmatrix}
=\]
\[
\begin{pmatrix}
e^h \left(1 - \frac{\alpha \beta}{2}\right) & e^h \left(1 - \frac{\alpha \beta}{2}\right) \beta \\
\alpha e^h \left(1 - \frac{\alpha \beta}{2}\right) & \alpha e^h \left(1 - \frac{\alpha \beta}{2}\right) \beta + e^h \left(1 - \frac{\alpha \beta}{2}\right)
\end{pmatrix}.
\]
Thus we have the following parametrization of the group element:

\[
g(h, \alpha, \beta) = \begin{pmatrix}
e^h \left( 1 - \frac{\alpha \beta}{2} \right) & e^h  \beta \\
e^h\alpha  & e^h \left( 1 +\frac{\alpha \beta}{2}\right) 
\end{pmatrix},
\]
where $h \in \Lzero$, $\alpha, \beta \in \Lone$. 

We note that since we are working over $\mathbb{C}$, we have the equivalence
$$
g(h, \alpha, \beta)=g(h+2\pi i n, \alpha, \beta),\quad n\in \mathbb{Z},
$$
however, for calculations and notation it is beneficial to keep additive variable.

\begin{Prop}
\label{prop:group}
The product $g(h_1, \alpha_1, \beta_1) \cdot g(h_2, \alpha_2, \beta_2) = g(h, \alpha, \beta)$ satisfies:

\begin{equation}
\label{eq:group_product}
\begin{aligned}
\alpha &= \alpha_1 + \alpha_2, \\
\beta &= \beta_1 + \beta_2, \\
h &= h_1 + h_2 + \frac{1}{2} (\alpha_1 \beta_2 - \alpha_2 \beta_1).
\end{aligned}
\end{equation}

Moreover, the parametrization of the inverse element is:
\[
g(h, \alpha, \beta)^{-1} = g(-h, -\alpha, -\beta).
\]
\end{Prop}
We can immediately upgrade these formulas for $GL(1|1)$. Indeed, we can parametrize any element $\tilde{g}\in GL(1|1)$ in the following way:
\[
\tilde{g} = \tilde{g}(h, s, \alpha, \beta) = g(h, \alpha, \beta) H_{s},
\]
where $g(h, \alpha, \beta) \in \SL(1|1)$ as above,
\[
H_{s} = \begin{pmatrix} e^{s/2} & 0 \\ 0 & e^{-s/2} \end{pmatrix},
\]
where $s$ is defined up to $4\pi i n$, $n\in \mathbb{Z}$. 
We can update the proposition above in the following way:

\begin{Prop}
\label{prop:groupproduct}
1)The group product for $\tilde{g}_1 = \tilde{g}(h_1, s_1, \alpha_1, \beta_1)$, $\tilde{g}_2 = \tilde{g}(h_2, s_2, \alpha_2, \beta_2)$ is:
\[
\tilde{g}_1 \cdot \tilde{g}_2 = \tilde{g}(h, s, \alpha, \beta),
\]
where:
\[
\alpha = \alpha_1 + e^{-s_1} \alpha_2, \quad \beta = \beta_1 + e^{s_1} \beta_2, \quad s = s_1 + s_2,
\]
\[
h = h_1 + h_2 + \frac{1}{2} \left( \alpha_1 e^{s_1} \beta_2 - e^{-s_1} \alpha_2 \beta_1 \right).
\]
2)The Berezinian of $\tilde{g}$ is:
$\sdet(\tilde{g}) = e^{s}$.\\
3)Given $\tilde{g}=\tilde{g}(h, s, \alpha, \beta)$, the inverse is:
\[
\tilde{g}^{-1} = \tilde{g}(-h, -s, -e^s \alpha, -e^{-s} \beta),
\]
\end{Prop}

\begin{proof}
Consider the product  
\[
\tilde{g}_1 \cdot \tilde{g}_2 = g(h_1, \alpha_1, \beta_1) H_{s_1} g(h_2, \alpha_2, \beta_2) H_{s_2}.
\]
and conjugation gives the following result
\[
H_{s_1} G(h_2, \alpha_2, \beta_2) H_{s_1}^{-1} = G(h_2, e^{-s_1} \alpha_2, e^{s_1} \beta_2).
\]
Now, applying the $\SL(1|1)$ product with scaled $\alpha, \beta$ we obtain the result. The Berezinian formula directly follows from its product properties. 
Let us solve $\tilde{g}(h, s, \alpha, \beta) \cdot \tilde{g}(h', s', \alpha', \beta') = I$. By Proposition \ref{prop:groupproduct}:
\[
\alpha_1 + e^{-s_1} \alpha_2 = 0, \quad \beta_1 + e^{s_1} \beta_2 = 0, \quad s_1 + s_2 = 0,
\]
\[
h_1 + h_2 + \frac{1}{2} (\alpha_1 e^{s_1} \beta_2 - e^{-s_1} \alpha_2 \beta_1) = 0.
\]
Thus: $\tilde{g}^{-1} = \tilde{g}(-h, -s, -e^{s} \alpha, -e^{-s} \beta).$
\end{proof}

\section{Transition Functions for $GL(1|1)$ bundles and cocycle conditions}

\subsection{$\SL(1|1)$ bundles} 
In this section we will describe the moduli space of $\SL(1|1)$-bundles on $C$. First, let us parametrize the transition function for $SL(1|1)$-bundle in the following way:
\[
g_{ij} = g(h_{ij}, \alpha_{ij}, \beta_{ij}),
\]
with $h_{ij} \in H^0(U_i \cap U_j, \mathcal{O}_C \otimes \Lzero), \alpha_{ij}, \beta_{ij} \in H^0(U_i \cap U_j, \mathcal{O}_C \otimes \Lone).
$ Then we have the following theorem characterizing these collections.

\begin{Thm}
\label{thm:cocycle}
1)  $\{h_{ij}\}$, $\{\alpha_{ij}\}$, 
$\{\beta_{ij}\}$ 
define  \v{C}ech 1-cochains on $C$.

2) Moreover,  $\{\alpha_{ij}\},\{\beta_{ij}\}$  
are representatives of  ${{H}}^1(C,\mathcal{O}\otimes\Lambda_1)$.

3)The cocycle condition for 
$\{g_{ij}\}$ implies the following condition on $\{h_{ij}\}$:
\begin{equation}
\label{eq:cocycle}
h_{ik} = h_{ij} + h_{jk} + \frac{1}{2} (\alpha_{ij} \beta_{jk} - \alpha_{jk} \beta_{ij}) +2\pi i n_{ijk},
\end{equation}
where $n_{ijk}\in \mathbb{Z}$.
\end{Thm}
\begin{proof}
The inverse 
$g(h_{ij}, \alpha_{ij}, \beta_{ij})^{-1}= g(-h_{ij}, -\alpha_{ij}, -\beta_{ij})=g(h_{ji}, \alpha_{ji}, \beta_{ji})$ proves point 1) of the proposition.

From Proposition \ref{prop:group}, 
the product $g(h_{ij}, \alpha_{ij}, \beta_{ij}) \cdot g(h_{jk}, \alpha_{jk}, \beta_{jk}) = g(h_{ik}, \alpha_{ik}, \beta_{ik})$ 
gives rise to the equations
\[
\begin{aligned}
\alpha_{ik} &= \alpha_{ij} + \alpha_{jk}, \\
\beta_{ik} &= \beta_{ij} + \beta_{jk}, \\
h_{ik}&= h_{ij} + h_{jk} + \frac{1}{2} (\alpha_{ij} \beta_{jk} - \alpha_{jk} \beta_{ij})+2\pi i n_{ijk},
\end{aligned}
\]
where $n_{ijk}\in \mathbb{Z}$
yielding 2), and 3) of the proposition. 
\end{proof}

\begin{Rem}We remark here that the body of the elements $e^{h_{ij}}$ gives an element of the Picard group $\check{H}^1(C,\mathcal{O}^*)$ and in the additive form the standard cocycle condition should be supplied with extra 
integers  $n_{ijk}$, which constitute the 2-cocycle corresponding to $\check{H}^2(C, \mathbb{Z})=\mathbb{Z}$, representing degree. For our calculations, which constitute Grassmann algebra corrections additive calculations are most suitable.
\end{Rem}

Now let us define the collection $\{g_{ijk}\}$, so 
that $g_{ijk}\in H^0(U_i\cup U_j\cup U_k)$ and 
\begin{equation}
\label{eq:2cocycle}
g_{ijk} = \frac{1}{2} (\alpha_{ij} \beta_{jk} - \alpha_{jk} \beta_{ij}).
\end{equation}
Then 
\begin{Prop}
\label{prop:cocycle2}
The collection $\{g_{ijk}\}$ forms a \v{C}ech 2-cochain, moreover it is a 2-cocycle, meaning that it satisfies relations $g_{jik} = -g_{ijk}$, $g_{kji} = -g_{ijk}$, $g_{kij} = g_{ijk}$, and $(\delta g)_{ijkl} = g_{jkl} - g_{ikl} + g_{ijl} - g_{ijk} = 0$.
\end{Prop}

\begin{proof}
The cocycle $g$ is in fact the cup product (see, e.g., \cite{Stacks}) of \v{C}ech cycles $\{\alpha_{ij}\}$, $\{\beta_{ij}\}$.

Indeed, using $\alpha_{ji} = -\alpha_{ij}$, $\beta_{ji} = -\beta_{ij}$, $\alpha_{kj} = -\alpha_{jk}$, $\beta_{kj} = -\beta_{jk}$, $\alpha_{ik} = \alpha_{ij} + \alpha_{jk}$, $\beta_{ik} = \beta_{ij} + \beta_{jk}$ we have: 
\[
g_{jik} = \frac{1}{2} (\alpha_{ji} \beta_{ik} - \alpha_{ik} \beta_{ji}) = \frac{1}{2} \left( (-\alpha_{ij}) (\beta_{ij} + \beta_{jk}) - (\alpha_{ij} + \alpha_{jk}) (-\beta_{ij}) \right).
\]
\[
= -\frac{1}{2} (\alpha_{ij} \beta_{jk} - \alpha_{jk} \beta_{ij}) = -g_{ijk}.
\]
as well as:
\[
g_{kji} = \frac{1}{2} (\alpha_{kj} \beta_{ji} - \alpha_{ji} \beta_{kj}) = -\frac{1}{2} (\alpha_{ij} \beta_{jk} - \alpha_{jk} \beta_{ij}) = -g_{ijk}.
\]
Thus indeed $g_{ijk}$ is a \v{C}ech cochain and we are ready to compute the coboundary:
\[
(\delta g)_{ijkl} = g_{jkl} - g_{ikl} + g_{ijl} - g_{ijk}.
\]
\[
g_{jkl} = \frac{1}{2} (\alpha_{jk} \beta_{kl} - \alpha_{kl} \beta_{jk}), \quad g_{ikl} = \frac{1}{2} (\alpha_{ik} \beta_{kl} - \alpha_{kl} \beta_{ik}),
\]
\[
g_{ijl} = \frac{1}{2} (\alpha_{ij} \beta_{jl} - \alpha_{jl} \beta_{ij}), \quad g_{ijk} = \frac{1}{2} (\alpha_{ij} \beta_{jk} - \alpha_{jk} \beta_{ij}).
\]
Let us use again cocycle conditions:
\[
\alpha_{ik} = \alpha_{ij} + \alpha_{jk}, \quad \alpha_{jl} = \alpha_{jk} + \alpha_{kl}, \quad \alpha_{il} = \alpha_{ij} + \alpha_{jl},
\]
\[
\beta_{ik} = \beta_{ij} + \beta_{jk}, \quad \beta_{jl} = \beta_{jk} + \beta_{kl}, \quad \beta_{il} = \beta_{ij} + \beta_{jl}.
\]
\[
g_{ikl} = \frac{1}{2} ((\alpha_{ij} + \alpha_{jk}) \beta_{kl} - \alpha_{kl} (\beta_{ij} + \beta_{jk})) = \frac{1}{2} (\alpha_{ij} \beta_{kl} + \alpha_{jk} \beta_{kl} - \alpha_{kl} \beta_{ij} - \alpha_{kl} \beta_{jk}),
\]
\[
g_{ijl} = \frac{1}{2} (\alpha_{ij} (\beta_{jk} + \beta_{kl}) - (\alpha_{jk} + \alpha_{kl}) \beta_{ij}) = \frac{1}{2} (\alpha_{ij} \beta_{jk} + \alpha_{ij} \beta_{kl} - \alpha_{jk} \beta_{ij} - \alpha_{kl} \beta_{ij}).
\]
Thus,
\begin{eqnarray}
&&(\delta g)_{ijkl} = \frac{1}{2} [ (\alpha_{jk} \beta_{kl} - \alpha_{kl} \beta_{jk}) - (\alpha_{ij} \beta_{kl} + \alpha_{jk} \beta_{kl} - \alpha_{kl} \beta_{ij} - \alpha_{kl} \beta_{jk}) +\nonumber\\
&&(\alpha_{ij} \beta_{jk} + \alpha_{ij} \beta_{kl} - \alpha_{jk} \beta_{ij} - \alpha_{kl} \beta_{ij}) - (\alpha_{ij} \beta_{jk} - \alpha_{jk} \beta_{ij})]=0.\nonumber
\end{eqnarray}
\end{proof}
Since  $H^2(C,\mathcal{O}_C)$ is trivial $\{g_{ijk}\}$ is exact and we can  
denote $\{f_{ij}\}$, where $f_{ij}= f_{ij}(\{\alpha_{ij}\}, \{\beta_{ij}\})\}$ is the solution of the equation 
\begin{equation}
\label{eq:gf}
g_{ijk}=\delta(f)_{ijk}.
\end{equation}
The condition on $\{h_{ij}\}$ immediately implies that 
$$e^{k_{ij}}=e^{h_{ij}+f_{ij}}$$
is a multiplicative \v{C}ech 1-cocycle 
determining the element of ${\check{H}}^1(C, \mathcal{O}^*)$. It depends on the choice of $f_{ij}$, defined up to $1$-cocycle.

The following statement shows that the bundle equivalences correspond to adding exact terms to cocycles $k_{ij}, \alpha_{ij}, \beta_{ij}$.

\begin{Prop}
There exists a solution to equation (\ref{eq:gf}), so that $SL(1|1)$ bundle equivalences:
\[
g_{ij}(h,\alpha,\beta) \to g_{ij}(h',\alpha', \beta')=r_ig_{ij}r^{-1}_{j},
\]
where $r_i=g(u_i, \lambda_i, \rho_i)$, so that $u_{i} \in H^0(U_i,\mathcal{O}_C \otimes \Lzero), \lambda_{i}, \rho_{i} \in H^0(U_i, \mathcal{O}_C \otimes \Lone) 
$ lead to equivalent cocycles:
\[ 
[k_{ij}]\sim [k'_{ij}], \quad [\alpha_{ij}]\sim [\alpha'_{ij}], \quad  [\beta_{ij}]\sim [\beta'_{ij}].
\]
\end{Prop}

\begin{proof}
We can verify it independently for the transformations related to various 1-parametric groups. For $r_j$ diagonal it is trivial. Let us assume 
$$r_j=g(1, \lambda_j,0).$$ 
Then \[\alpha'_{ij}=\alpha_{ij}+\lambda_i-\lambda_j, \quad \beta'_{ij}=\beta_{ij}, \quad h'_{ij}=h_{ij}+\frac{1}{2}(\lambda_i+\lambda_j)\beta_{ij}\].
At the same time, 
\[
g'_{ijk}=g_{ijk}+\frac{1}{2} ((\lambda_{i}-\lambda_j) \beta_{jk} - (\lambda_{j}-\lambda_{k}) \beta_{ij})=g_{ijk}+\frac{1}{2}(\lambda_{i}\beta_{jk}-\lambda_j\beta_{ik}+\lambda_k\beta_{ij})=g_{ijk}+\hat{g}_{ijk},
\]
where we denoted the correction $\hat{g}_{ijk}=\frac{1}{2}(\lambda_{i}\beta_{jk}-\lambda_j\beta_{ik}+\lambda_k\beta_{ij})$. 

Introducing $\hat{h}_{ij}=\frac{1}{2}(\lambda_i+\lambda_j)\beta_{ij}$, we have:
\[
\delta(\hat{h})_{ijk}=\frac{1}{2}(\lambda_i+\lambda_j)\beta_{ij}-\frac{1}{2}(\lambda_i+\lambda_k)\beta_{ik}+
\frac{1}{2}(\lambda_j+\lambda_k)\beta_{jk}=\frac{1}{2}(\lambda_i\beta_{kj}+\lambda_j\beta_{ik}+\lambda_k\beta_{ji}).\]
Thus we see that $\hat{g}_{ijk}=-\delta(\hat{h})_{ijk}$, and therefore choosing $f'_{ij}=f_{ij}+\hat{f}_{ij}$, so that  $\delta(\hat{f})_{ijk}=-\delta(\hat{h})_{ijk}$ and 
\[
k'_{ij}=h'_{ij}+f'_{ij}=k_{ij}+\hat{h}_{ij}+\hat{f}_{ij}.
\]
One can show that we can choose the solution of the equation (\ref{eq:gf}) to that $\hat{f}_{ij}=-\hat{h}_{ij}$. Let us fix the solution $\{f_{ij}\}$ of the 
equation (\ref{eq:gf}) for 
the chosen representatives $\{\alpha_{\ij}\}$ and $\{\beta_{ij}\}$ of 
$[\alpha_{\ij}]$ and $[\beta_{ij}]$. For  
$\alpha'_{ij}=\alpha_{ij}+\lambda_i-\lambda_j$, we extend it by $f'_{ij}=f_{ij}-\frac{1}{2}(\lambda_i+\lambda_j)\beta_{ij}$. Therefore 
$\{k'_{ij}\}$ does indeed belong to the same cohomology class as $\{k_{ij}\}$.
We can now repeat the same construction for $r_j=g(1,0,\lambda_j)$, which fixes solution for $f_{ij}$.  
\end{proof}
As a consequence, we obtain the following Theorem.
\begin{Thm}
\label{thm:paramsl11}
Moduli space of $\SL(1|1)$-bundles is isomorphic to  
$\Pic_{\Lambda}(C) \times H^1(C, \mathcal{O}_C) \otimes \Lone \times H^1(C, \mathcal{O}_C) \otimes \Lone$, where by $\Pic_{\Lambda}(C)$ we understand the space $H^1(C, \mathcal{O}^*\otimes \Lambda_0^{\times})$.
\end{Thm}

In the following we will refer to the {\it degree} of $\SL(1|1)$ bundle as a degree of line bundle as degree of the corresponding component of the body of its moduli space, $Pic(C)$.

\subsection{Generalization to $GL(1|1)$}

Transition functions on $U_i \cap U_j$ are:
\[
\tilde{g}_{ij} = \tilde{g}(h_{ij}, s_{ij}, \alpha^i_{ij}, \beta^i_{ij}) = g(h_{ij}, \alpha^i_{ij}, \beta^i_{ij}) H_{s_{ij}},
\]
where $g(h_{ij}, \alpha^i_{ij}, \beta^i_{ij}) \in \SL(1|1)$, as in previous subsection. 
\[
H_{s_{ij}} = \begin{pmatrix} e^{s_{ij}/2} & 0 \\ 0 & e^{-s_{ij}/2} \end{pmatrix},
\]
so that $h_{ij}, s_{ij}: U_i \cap U_j \to \mathcal{O}_C \otimes \Lzero$ are even coordinates, $\alpha^i_{ij}, \beta^i_{ij}: U_i \cap U_j \to \mathcal{O}_C \otimes \Lone$ are odd coordinates on $U_i\cup U_j$. 
%These define classes in $H^1(C, L^{\pm 1} \otimes \mathcal{O}_C \otimes \Lone)$ and $H^1(C, \mathcal{O}_C \otimes \Lzero)$ as we shall see below for a determinantal bundle $L$.

The Berezinian of $\tilde{g}_{ij}$ is:
\[
\sdet(\tilde{g}_{ij}) = e^{s_{ij}},
\]
defining a line bundle $L \in \Pic(C)$ with $[e^{s_{ij}}] \in H^1(C, \mathcal{O}^*_C \otimes \Lambda_0^\times)$. 

By Proposition \ref{prop:groupproduct} we have:
\begin{Thm}
\label{thm:transcocycle}
The parametrization of $\tilde{g}$ satisfies the following properties:
%$\alpha^i_{ij}, \beta^i_{ij}: U_i \cap U_j \to \mathcal{O}_C \otimes \Lone$, $h_{ij}, s_{ij}: U_i \cap U_j \to \mathcal{O}_C \otimes \Lzero$, the inverse is:
%\[
%\tilde{g}_{ij}^{-1} = \tilde{g}(-h_{ij}, -s_{ij}, -e^{s_{ij}} \alpha^i_{ij}, -e^{-s_{ij}} \beta^i_{ij}),
%\]
\[h_{ij}=-h_{ji},\quad  s_{ij}=-s_{ji}, \quad \alpha^j_{ji}=-e^{s_{ij}} \alpha^i_{ij}, \quad \beta^j_{ji}=-e^{-s_{ij}} \beta^i_{ij}\]
On triple intersections we have: $U_i \cap U_j \cap U_k$:
\[
\alpha^i_{ik} = \alpha^i_{ij} + e^{-s_{ij}} \alpha^j_{jk}, \quad \beta^i_{ik} = \beta^i_{ij} + e^{s_{ij}} \beta^j_{jk},
\]
\[
h_{ik} = h_{ij} + h_{jk} + \frac{1}{2} \left( \alpha^i_{ij} e^{s_{ij}} \beta^j_{jk} - e^{-s_{ij}} \alpha^j_{jk} \beta^i_{ij} \right),
\]
defining in particular $[\alpha^i_{ij}] \in H^1(C, L^{-1} \otimes \Lone)$, $[\beta^i_{ij}] \in H^1(C, L \otimes \Lone)$.
\end{Thm}

Defining the collection $\{\tilde{g}_{ijk}\}$ 
\begin{equation}
\label{eq:2cocycle}
\tilde{g}_{ijk} = \frac{1}{2} \left( \alpha^i_{ij} e^{s_{ij}} \beta^j_{jk} - e^{-s_{ij}} \alpha^j_{jk} \beta^i_{ij} \right).
\end{equation}
of $H^0(U_i\cup U_j\cup U_k)$ we obtain 
\begin{Prop}
\label{prop:cocycle2}
The collection $\{\tilde{g}_{ijk}(\{{\tilde{s}_{rs}\}, \{\alpha_{rs}}\}, \{\beta_{rs}\})\}$ forms a \v{C}ech 2-cocycle.
\end{Prop}

The proof is the same as in $\SL(1|1)$ case, now accompanied by the twist parameters $s_{ij}$. We leave the proof to the reader. 
Then the theorem from previous section can be generalized in the following way.

\begin{Thm}
\label{thm:paramsl11}
The moduli space of $GL(1|1)$-bundles for a fixed Berezinian bundle $L$ are parameterized by 
$\Pic_{\Lambda}(C) \times H^1(C, L \otimes \Lone) \times H^1(C, L^{-1} \otimes \Lone)$.
\end{Thm}

In the following we will refer to the {\it degree} of $\GL(1|1)$ bundle as a 
pair of integers $(m,n)$, where $m$ is the degree of a Berezinian bundle, while the second degree is the degree of the corresponding component of the body of the above  moduli space: $Pic(C)$.

\section {Higgs field}
The $GL(1|1)$ Higgs field, namely an even section of $End(\mathcal{E})\otimes K_C$ is locally represented as the matrix:
\begin{eqnarray}
\label{eq:higgsgen}
\Phi = \begin{pmatrix} a & \beta \\ \gamma & d \end{pmatrix},
\end{eqnarray}
with $a, d \in H^0(U_i, K_C \otimes \Lzero)$, $\beta, \gamma \in H^0(U_i, K_C \otimes \Lone)$, and transforms as:
\[
\Phi_j = \tilde{g}_{ij}^{-1} \Phi_i \tilde{g}_{ij}.
\]
The supertrace is:
\[
\str(\Phi) = a - d \in H^0(C, K_C\otimes \Lambda_0),
\]

which is a holomorphic 1-form on $C$.

% Explaining the behavior of the supertrace
\subsection{Behavior of the Supertrace}

The supertrace $\str(\Phi) \in H^0(C, K_C\otimes \Lambda_0)$ is a holomorphic 1-form, and its behavior depends on the genus $g$:

\begin{itemize}
    \item \textbf{For $g \geq 2$}: The canonical bundle $K_C$ has degree $2g - 2 > 0$, and a non-zero holomorphic 1-form $\str(\Phi)$ has exactly $2g - 2$ zeros, counted with multiplicity. Thus, if $\str(\Phi) \neq 0$, it cannot be non-zero everywhere; it must vanish at $2g - 2$ points (or fewer with higher multiplicity).
    \item \textbf{For $g = 1$}: The canonical bundle $K_C$ is trivial, so $\str(\Phi)$ can be a non-zero constant 1-form (e.g., $dz$ in local coordinates), which is non-vanishing everywhere.
    \item \textbf{For $g = 0$}: Since $K_C = \mathcal{O}(-2)$, we have $H^0(C, K_C) = 0$, so $\str(\Phi) = 0$ everywhere.
\end{itemize}

This behavior influences the classification, particularly when $\str(\Phi) \neq 0$, as the zeros of $\str(\Phi)$ affect the structure of the Higgs field and the splitting of the bundle.

\subsection{Diagonalization of Higgs field}
Let $\Phi$ be a supermatrix parametrized as in (\ref{eq:higgsgen}). Then we have the following proposition regarding eigenvalue properties.
\begin{Prop}
\label{thm:eigen}
If $\str(\Phi) = a - d \neq 0$, $\Phi$ has eigenvalues $\lambda_+ = a - \frac{\beta \gamma}{a - d}$, $\lambda_- = d - \frac{\beta \gamma}{a - d}$, 
or in other words:
\begin{equation}
\lambda_{\pm}=\frac{1}{2}\Big[\frac{\str({\Phi}^2)}{\str({\Phi})}\pm{ \str({\Phi})}\Big]
\end{equation}
with eigenvectors $\begin{pmatrix} 1 \\ \frac{{\gamma}}{\str({\Phi}) } \end{pmatrix}$, $\begin{pmatrix} -\frac{\beta}{\str({\Phi}) }  \\ 1 \end{pmatrix}$.
\end{Prop}

\begin{proof}
We start by solving $\Phi v = \lambda v$, where $v = \begin{pmatrix} x \\ y \end{pmatrix}$:

\[
\begin{pmatrix}
a x + \beta y \\
\gamma x + d y
\end{pmatrix} = \lambda \begin{pmatrix}
x \\
y
\end{pmatrix}.
\]
This gives:
\[
(a - \lambda) x + \beta y = 0, \quad \gamma x + (d - \lambda) y = 0.
\]
For $\lambda_+ = a - \frac{\beta \gamma}{a - d}$:
\begin{equation}
\label{eq:eigen1}
\beta y = \frac{\beta \gamma}{a - d} x, \quad y = \frac{\gamma}{a - d} x.
\end{equation}
and choosing $x = 1$ we obtain:
\[
v_+ = \begin{pmatrix} 1 \\ \frac{\gamma}{a - d} \end{pmatrix}.
\]
For $\lambda_- = d - \frac{\beta \gamma}{a - d}$ we have:
\begin{equation}
\label{eq:eigen2}
(a - d) x + \beta y = -\frac{\beta \gamma}{a - d} x, \quad x = -\frac{\beta}{a - d} y.
\end{equation}
and choosing $y = 1$:
\[
v_- = \begin{pmatrix} -\frac{\beta}{a - d} \\ 1 \end{pmatrix}.
\]
Now given that $\lambda_+^2-\lambda_-^2=a^2-d^2=\str(\Phi^2)$, $\lambda_+-\lambda_-=a-d=\str(\Phi)$ and thus $\lambda_++\lambda_-=\frac{\str(\Phi^2)}{\str(\Phi)}$, we obtain the expressions for eigenvalues in invariant form.\end{proof}

Therefore, if $\str(\Phi) \neq 0$, there exists $P \in \GL(1|1)$ such that:

\[
P^{-1} \Phi P = \begin{pmatrix}
\lambda_1 & 0 \\
0 & \lambda_2
\end{pmatrix}, \quad {\rm where}\quad 
P = \begin{pmatrix} 1 & -\frac{\beta}{\str(\Phi)} \\ \frac{\gamma}{\str(\Phi)} & 1 \end{pmatrix}.
\]
Unfortunately, the diagonalization procedure is meromorphic when $\str({\Phi})$ has zeroes. However, for genus 1 we obtain that as long as $\str({\Phi})\in \Lambda^{\times}$, the Higgs bundle breaks into the sum of line bundles. In the next section we present a more general description of Higgs bundles.

%This simple superlinear algebra exercise shows that in the case of $\Phi$ has supertrace not identically trivial, the resulting bundle is meromorphically split, moreover, one can characterize the corresponding Higgs bundles via globally defined meromorphic differentials $\lambda_{\pm}\in H^0(C,K[D])$, where $D$ is the divisor corresponding to $\str{\Phi}$ and two elements of picard group.

%In the next section we present a more general description of Higgs bundles.

\subsection{Transformations of Higgs field}

Let us parameterize locally Higgs field in the following way:

\[
\Phi_i = \begin{pmatrix}
a_i + \frac{b_i}{2} & \delta_i \\
\gamma_i & a_i - \frac{b_i}{2}
\end{pmatrix},
\]

where $a_i, b_i \in H^0(U_i, K_C \otimes \Lzero)$, $\delta_i, \gamma_i \in H^0(U_i, K_C \otimes \Lone)$, $\str(\Phi_i) = b_i$.

The transition function is $\tilde{g}_{ij}(s_{ij}, h_{ij}, \alpha_{ij}, \beta_{oij}) = N_- H N_+ H_s$, where:

\[
N_- = \begin{pmatrix} 1 & 0 \\ \alpha^i_{ij} & 1 \end{pmatrix}, \quad H = \begin{pmatrix} l_{ij} & 0 \\ 0 & l_{ij} \end{pmatrix}, \quad l_{ij} = e^{h_{ij}} \left(1 - \frac{\alpha^i_{ij} \beta^i_{ij}}{2}\right),
\]

\[
N_+ = \begin{pmatrix} 1 & \beta^i_{ij} \\ 0 & 1 \end{pmatrix}, \quad H_s = \begin{pmatrix} e^{s_{ij}/2} & 0 \\ 0 & e^{-s_{ij}/2} \end{pmatrix}.
\]

This results in the following formula:

\[
\Phi_j = \begin{pmatrix}
a_i + \frac{b_i}{2} + \delta_i \alpha^i_{ij} - \beta^i_{ij} (\gamma_i - \alpha^i_{ij} b_i) & e^{-s_{ij}} \left(\delta_i + \beta^i_{ij} b_i\right) \\
e^{s_{ij}} \left(\gamma_i - \alpha^i_{ij} b_i\right) & a_i - \frac{b_i}{2} + \delta_i\alpha^i_{ij}  -\beta^i_{ij} \left(\gamma_i - \alpha^i_{ij} b_i\right)
\end{pmatrix}.
\]

\subsection{$\mathfrak{sl}(1|1)$ Higgs field} 
Let us first focus on the case of $\mathfrak{sl}(1|1)$ Higgs field, namely the one with zero supertrace: $b_i=0$, so that explicitly we have:
\begin{eqnarray}
\Phi_i &=& \begin{pmatrix}
a_i  & \delta_i \nonumber\\
\gamma_i & a_i 
\end{pmatrix}, \\
\Phi_j &=& \tilde{g}_{ij}^{-1} \Phi_i \tilde{g}_{ij}=\begin{pmatrix}
a_i + \delta_i \alpha^i_{ij} - \beta^i_{ij} \gamma_i  & e^{-s_{ij}}\delta_i \\
e^{s_{ij}} \gamma_i  & a_i  + \delta_i\alpha^i_{ij}  -\beta^i_{ij} \gamma_i 
\end{pmatrix}
\end{eqnarray}

Since the above transformation yields 
\[
\delta_j = e^{s_{ji}}\delta_i, \quad \gamma_j = e^{-s_{ji}}\gamma_i,
\]
we obtain that $\{\delta_i\}$, $\{\gamma_i\}$ are elements of $\check{H}^0(C,K\otimes L)\otimes \Lone$, $\check{H}^0(C,K\otimes L^{-1})\otimes \Lone$ correspondingly.
At the same time, we have the following equation for $a_i$:
\[
a_j = a_i + \delta_i \alpha^i_{ij} -  \beta^i_{ij}\gamma_i = a_i + t_{ij},
\]
\[
t_{jk} - t_{ik} + t_{ij} = (\delta _j\alpha^j_{jk} - \beta^j_{jk}\gamma_j) - (\delta_i \alpha^i_{ik} - \beta^i_{ik}\gamma_i ) + (\delta_i \alpha^i_{ij} -  \beta^i_{ij}\gamma_i) = 
\]
\[
 (\delta _ie^{s_{ji}}\alpha^j_{jk} - e^{-s_{ji}} \beta^j_{jk}\gamma_i) - (\delta_i \alpha^i_{ik} -  \beta^i_{ik}\gamma_i) + (\delta_i \alpha^i_{ij} -  \beta^i_{ij}\gamma_i)=0,
\]
since $\alpha^i_{ik} = \alpha^i_{ij} + e^{-s_{ij}} \alpha^j_{jk}$, $\beta^i_{ik} = \beta^i_{ij} + e^{s_{ij}} \beta^j_{jk}$. Thus, $[t_{ij}] \in H^1(C, K_C \otimes \Lzero) \cong \Lzero$, where $H^1(C, K_C) \cong \mathbb{C}$.

% Step 4: Constraint [f_{ij}] = 0
The constraint $[t_{ij}] = 0$ ensures $\Phi$ is globally defined. We have then
\[
[\{t_{ij}\}] = [\delta_i \alpha^i_{ij} - 
 \beta^i_{ij}\gamma_i] = 
[\{\delta_i\}] [\{\alpha^i_{ij}\}] -  [\{\beta^i_{ij}\}] [\{\gamma_i\}]= 0 \in H^1(C, K_C) \otimes \Lzero.
\]
This imposes the condition with a shorthand notation:
\[
[\delta]\cdot [\alpha] = [\beta]\cdot[\gamma],
\]
where dot stands for the cup product, restricting $[\alpha_{ij}]\in \check{H}^1(C, L^{-1}\otimes \Lone), [\beta_{ij}] \in \check{H}^1(C, L\otimes \Lone)$, $\delta\in H^0(C, K_C\otimes L\otimes \Lone)$ , $\gamma \in \check{H}^0(C, K_C\times L^{-1} \otimes \Lone)$. When satisfied, there exists $\eta_i \in \check{H}^0(U_i, K_C \otimes \Lzero)$ such that:
\[
\delta_i \alpha^i_{ij} - \gamma_i \beta^i_{ij} = \eta_j - \eta_i,
\]
allowing the choice of local sections:
\[
a_i = \eta_i + a, \quad a \in H^0(C, K_C \otimes \Lzero) \cong \mathbb{C}^g \otimes \Lzero,
\]
so that: $a_j - a_i = (\eta_j + a) - (\eta_i + a) = \eta_j - \eta_i = t_{ij}$.

Thus we are ready to formulate a theorem classifying Higgs bundles for $GL(1|1)$.

\begin{Thm}
Let $C$ be a compact Riemann surface. Moduli space of $\GL(1|1)$ Higgs bundles $(E,\Phi)$ on $C$ with a fixed Berezinian bundle $\sdet(E)=L$ and the condition $\str(\Phi)\equiv 0$ is parametrized by the following data:
%\begin{enumerate}
%\item If 1-form $\str(\Phi)\in H^{0}(C,K_C)$ is a nonzero holomorphic 1-form, and $\sdet(E)=L$ then the moduli space of such $\GL(1|1)$ Higgs bundles is
%isomorphic to $(\mathcal{L}, \phi_1, \phi_2)$ of %$Pic_{\Lambda_0}(C)\times H^0(C, K)\times H^0(C, K^2)$,  parametrized by the elements $(\mathcal{L}, \phi_1, \phi_2)$, where $\mathcal{L}\in Pic_{\Lambda_0}(C)$, $\phi_1=\str(\Phi)$, $\phi_2=\str(\Phi^2)$.

\begin{equation}
(\mathcal{L}, a, [\alpha],[\beta], [\delta], [\gamma])\in Pic_{\Lambda}(C)\times H^0(C, K_C\otimes \Lambda_0)\times M\times M^*,
\end{equation}
where $$M=H^1(C, {L}^{-1} \otimes \Lone)\times H^1(C, {L} \otimes \Lone), \quad M^*=H^0(C, {L}\otimes K \otimes \Lone)\times H^0(C, {L}^{-1}\otimes K \otimes \Lone)$$
and odd variables are satisfying a quadratic relation: \begin{equation}[\delta]\cdot[\alpha]=[\beta]\cdot [\gamma]\in H^2(C, \Lambda_0)\cong \Lambda_0,
\end{equation}
where dot stands for the standard cup product.
\end{Thm}
\subsection{General $\GL(1|1)$ Higgs bundles}
Nonvanishing supertrace makes the situation more involved. Given that 
\begin{eqnarray}
\Phi_i &=& \begin{pmatrix}
a_i+\frac{b_i}{2}  & \delta_i \nonumber\\
\gamma_i & a_i -\frac{b_i}{2}
\end{pmatrix}, \\
\Phi_j &=& \tilde{g}_{ij}^{-1} \Phi_i \tilde{g}_{ij}=\begin{pmatrix}
a_i + \frac{b_i}{2} + \delta_i \alpha^i_{ij} - \beta^i_{ij} (\gamma_i - \alpha^i_{ij} b_i) & e^{-s_{ij}} \left(\delta_i + \beta^i_{ij} b_i\right) \\
e^{s_{ij}} \left(\gamma_i - \alpha^i_{ij} b_i\right) & a_i - \frac{b_i}{2} + \delta_i\alpha^i_{ij}  -\beta^i_{ij} \left(\gamma_i - \alpha^i_{ij} b_i\right)
\end{pmatrix},\nonumber
\end{eqnarray}
where $\{b_i\}$ represents an element of $\check{H}^0(C, K_C)$, 
the $(1,2)$ and $(2,1)$ elements imply that  $[\{\alpha_{ij}^i\}][b_i]\in  H^1(C, L^{-1}\otimes \Lambda_1)$ and $[\{\beta_{ij}^i\}][b_i]\in  H^1(C, L^{-1}\otimes \Lambda_1)$ are trivial, namely
\begin{equation}\label{eq:brel}
\alpha^i_{ij} b_i=\gamma_i-e^{-s_{ij}}\gamma_j, \quad  \beta^i_{ij} b_i=e^{s_{ij}}\delta_j-\delta_i.
\end{equation}
Next, denoting 
\begin{eqnarray}\label{eq:ccoc}
c_{ij}=\beta^i_{ij}\gamma_i-\delta_i \alpha^i_{ij}  -\beta^i_{ij} \alpha^i_{ij} b_i
\end{eqnarray}
the condition for (1,1) and (2,2) elements give:
\begin{equation}
\label{eq:1exact}
c_{ij}=a_i-a_j.
\end{equation}
One can see that the left hand side is indeed a \v{C}ech 1-cochain using the relations (\ref{eq:brel}). At the same time one can show that $c_{ij}$ satisfies cocycle relation, namely vanishing of 
\begin{equation}\label{eq:coc}
c_{ij}+c_{jk}+c_{ki}.
\end{equation}
To do that let us transform all $\beta\gamma$ terms of (\ref{eq:coc}):
\begin{eqnarray}
&&\beta^i_{ij}\gamma_i+\beta^j_{jk}\gamma_j+\beta^k_{ki}\gamma_k=\beta^i_{ik}\gamma_i+e^{s_{ik}}\beta^k_{kj}\gamma_i+\beta^j_{jk}\gamma_j+\beta^k_{ki}\gamma_k=\nonumber\\
&&\beta^i_{ik}(\gamma_i-e^{s_{ki}}\gamma_k)+e^{s_{ik}}\beta^k_{kj}(\gamma_i-e^{s_{ji}}\gamma_j)=\nonumber\\
&&\beta^i_{ik}\alpha^{i}_{ik} b_i+e^{s_{ik}}\beta^k_{kj}
\alpha^i_{ij}b_i,
\end{eqnarray}
and then all $\delta\alpha$ terms of (\ref{eq:coc}):
\begin{eqnarray}
&&-\delta_k\alpha^k_{ki}-\delta_i\alpha^i_{ij}-\delta_j\alpha^j_{jk}=
-\delta_k\alpha^k_{kj}-e^{-s_{kj}}\delta_k \alpha^j_{ji}-\delta_i\alpha^i_{ij}-\delta_j\alpha^j_{jk}=\nonumber\\
&& (-\delta_k+e^{-s_{jk}}\delta_j)\alpha^k_{kj}+e^{-s_{kj}}(-\delta_k+e^{-s_{ik}}\delta_i)\alpha^{j}_{ji}=\nonumber\\
&&b_k\beta^k_{kj}\alpha^k_{kj}+e^{-s_{kj}}b_k\beta^k_{ki}\alpha^{j}_{ji}\nonumber
\end{eqnarray}
Adding these two expressions we have:
\begin{eqnarray}
&&b_i\beta^i_{ik}\alpha^{i}_{ik}+b_k\beta^k_{kj}\alpha^k_{kj}+e^{s_{ik}}\beta^k_{kj}
\alpha^i_{ij}b_i+e^{-s_{kj}}b_k\beta^k_{ki}\alpha^{j}_{ji}=\nonumber\\
&&b_i\beta^i_{ik}\alpha^{i}_{ik}+b_k\beta^k_{kj}\alpha^k_{kj}+(e^{s_{ik}}\beta^k_{kj}-e^{-s_{kj}}
e^{-s_{ji}}\beta^k_{ki})\alpha^i_{ij}b_i=\nonumber\\
&&b_i\beta^i_{ik}\alpha^{i}_{ik}+b_k\beta^k_{kj}\alpha^k_{kj}+(e^{s_{ik}}\beta^k_{kj}-e^{-s_{ki}}
\beta^k_{ki})\alpha^i_{ij}b_i=\nonumber\\
&&b_i\beta^i_{ik}\alpha^{i}_{ik}+b_k\beta^k_{kj}\alpha^k_{kj}+(e^{s_{ik}}\beta^k_{kj}+
\beta^i_{ik})\alpha^i_{ij}b_i=
b_i\beta^i_{ik}\alpha^{i}_{ik}+b_k\beta^k_{kj}\alpha^k_{kj}+b_i\beta^i_{ij}\alpha^i_{ij}
\end{eqnarray}
Now one can see that the result cancels out with $\beta\alpha b$ terms in (\ref{eq:coc}) as we desired.
Since $H^1(C, K_C\otimes \Lambda_0)\cong H^0(C,O_C\otimes \Lambda_0)=\Lambda_0$, we set to zero $[c_{ij}]=0$ based on (\ref{eq:1exact}). 

Now note that here the solutions for $\{\gamma_i\}$, 
$\{\delta_i\}$ to the equations (\ref{eq:brel}) are  
unique up to 
$[\tilde\gamma]:=\{\tilde{\gamma}_i\}\in \check{H}^0(C,K\otimes L^{-1}\otimes \Lone)$, 
and  $[\delta]=\{\tilde{\delta}_i\}\in\check{H}^0(C,K\otimes L\otimes \Lone)$, which according to equation  (\ref{eq:1exact}) have to satisfy familiar condition
\begin{equation}
[\tilde\delta]\cdot[\alpha]=[\beta]\cdot[\tilde\gamma]
\end{equation}
The same applies to the solution of the equation (\ref{eq:1exact}) for $a_i$: it is unique to the elements $[\tilde{a}_i]\in \check{H}^0(C, K_C\otimes \Lambda_0)$.
Thus we can update the classification theorem.

\begin{Thm}
Let $C$ be a compact Riemann surface. The moduli space of $\GL(1|1)$ Higgs bundles $(E,\Phi)$ where  with the fixed Berezinian bundle $\sdet(E)=L$ and $\str({\Phi})=[b]\in H^0(C, K_C\otimes \Lambda_0)$ is parametrized by the following data:
%\begin{enumerate}
%\item If 1-form $\str(\Phi)\in H^{0}(C,K_C)$ is a nonzero holomorphic 1-form, and $\sdet(E)=L$ then the moduli space of such $\GL(1|1)$ Higgs bundles is
%isomorphic to $(\mathcal{L}, \phi_1, \phi_2)$ of %$Pic_{\Lambda_0}(C)\times H^0(C, K)\times H^0(C, K^2)$,  parametrized by the elements $(\mathcal{L}, \phi_1, \phi_2)$, where $\mathcal{L}\in Pic_{\Lambda_0}(C)$, $\phi_1=\str(\Phi)$, $\phi_2=\str(\Phi^2)$.

\begin{equation}
(\mathcal{L}, \tilde{a}, [\alpha],[\beta], [\tilde\delta],  [\tilde{\gamma}])\in Pic_{\Lambda}(C)\times H^0(C, K_C\otimes \Lambda_0)\times M\times M^*
\end{equation}
where $$M=H^1(C, {L}^{-1} \otimes \Lone)\times H^1(C, {L} \otimes \Lone), \quad M^*=H^0(C, {L}\otimes K \otimes \Lone)\times H^0(C, {L}^{-1}\otimes K \otimes \Lone)$$
where the variables are satisfying the following relations: 
\begin{equation}
[b]\cdot[\alpha]=[b]\cdot[\beta]=0, \quad [c]=0, \quad [\tilde{\delta}]\cdot[\alpha]=[\beta]\cdot [\tilde{\gamma}],
\end{equation}
where $[c]=[c](\alpha, \beta, b)\in H^1(C, K_C\otimes\Lambda_1)$ is defined in (\ref{eq:ccoc}), where dot stands for the standard cup product.
\end{Thm}

\section{Hitchin equations for $SL(1|1)$ bundles}

Given an $SL(1|1)$ bundle $E$ with underlying line bundle of degree 0 and a Higgs field $\Phi$, $\str(\Phi)=0$, consider the Hermitian metric parametrized locally as follows:
$$
H=g(u, \rho, \bar{\rho})=
\begin{pmatrix}
e^u \left( 1 -\frac{\rho \bar \rho}{2} \right) & e^u  \bar\rho \\
e^u\rho  & e^u \left( 1 +\frac{\rho \bar\rho}{2}\right) .
\end{pmatrix}
$$
Here $\rho\in \Lambda_1$, $u\in \Lambda_0$. We also note that here we use the complex conjugation as follows: for $u,v\in\Lambda$, $\bar{uv}=\bar{v}\bar{u}$.  

The Chern connection (see, e.g., \cite{Huybrechts}) associated to this Hermitian form is:
\begin{eqnarray}
\nabla=\partial +\bar{\partial} +H^{-1}\partial H,
\end{eqnarray}
where locally
\begin{eqnarray}
H^{-1}\partial H=
\begin{pmatrix}
\partial_z u-\frac{1}{2}\bar{\rho}\partial_z\rho-\frac{1}{2}{\rho}\partial_z\bar\rho & \partial_z\bar\rho \\
\partial_z\rho  & \partial_z u-\frac{1}{2}\bar{\rho}\partial_z\rho-\frac{1}{2}{\rho}\partial_z\bar\rho 
\end{pmatrix}
\end{eqnarray}
so that the curvature is:
$$
F=\bar{\partial}(H^{-1}\partial H).
$$
Therefore, equation $F=0$ locally is written locally as 
two equations, corresponding to diagonal and off-diagonal elements:
\begin{eqnarray}
&&\partial_{\bar z}\partial_z u-\frac{1}{2}(\partial_{\bar z}\bar \rho\partial_z\rho+\partial_{\bar z} \rho\partial_z\bar\rho)=0,\label{eq:0curv1}\\
&&\partial_z \partial_{\bar z}\rho=0 \label{eq:0curv2}
\end{eqnarray}
Notice that the patch to patch transformation of $H$ is given by:
\begin{eqnarray}\label{eq: metrictransform}
&&g(\bar h, \bar{\beta},\bar{\alpha})Hg(h, \alpha,\beta)
=\nonumber\\
&&g(\bar h, \bar{\beta},\bar{\alpha})g(u, \rho, \bar{\rho})g({h}, \alpha,\beta)=g(\bar{h}, \bar{\beta},\bar{\alpha})g(h+u+\frac{1}{2}(\rho\beta-\alpha\bar \rho), \alpha+\rho, \beta+\bar{\rho})=\nonumber\\
&&g(u+h+\bar{h}+\frac{1}{2}(\rho\beta-\alpha\bar \rho) +\frac{1}{2}(\bar{\beta}(\beta+\bar\rho)-(\alpha+\rho)\bar{\alpha}), \rho+\alpha+\bar\beta, \bar{\rho}+\beta+\bar{\alpha})=\nonumber\\
&&g(u+h+\bar{h}-\frac{1}{2}(\alpha\bar{\alpha}+\beta\bar{\beta}+\rho(\bar{\alpha}-\beta)+(\alpha-\bar{\beta})\bar{\rho}), \rho+\alpha+\bar\beta, \bar{\rho}+\beta+\bar{\alpha}).
\end{eqnarray}
Now solving locally the equation (\ref{eq:0curv2})
we obtain 
$$
\rho=\rho^h+\rho^a,
$$
which correspond to the decomposition into holomorphic and antiholomorphic function. Then solution to (\ref{eq:0curv1}) is: 
$$
u=v+\frac{1}{2}\bar{\rho^h}{\rho}^h+\frac{1}{2}{\rho}^a\bar{\rho^a},
$$
where $v$ is harmonic. 

Now in order to show the existence of the solution for $H$ we have to show that it is glued properly. We will show how it breaks into line bundle harmonic metric solutions.
It is indeed so for $\rho$. 
\begin{eqnarray}
&&\frac{1}{2}\rho_j+\bar\rho_j=\frac{1}{2}\rho_i+\bar\rho_i+{\rm Re}[\alpha_{ji}+\beta_{ji}],\nonumber\\
&&\frac{i}{2}( \rho_j-\bar\rho_j)=\frac{i}{2}( \rho_i-\bar\rho_i)+{\rm Re}[i(\alpha_{ji}-\beta_{ji})],
\end{eqnarray}
where $\rho_j$ is the solution on the equation (\ref{eq:0curv2}) on $U_j$ subset  of the covering.
Given that $\{\alpha_{ji}\}$, $\{\beta_{ji}\}$ are elements of $\check{H}^1(C,\mathcal{O}_C\otimes \Lambda_1)$, standard argument regarding Laplace operator and Chern class for degree 0 line bundle described by $\{\alpha_{ji}\}$, $\{\beta_{ji}\}$ shows the existence of $\rho_j$ (see, e.g.,\cite{rank1Higgs}).

Similarly, let us describe the transformation for $v$ solutions on $U_i\cup U_j$:
\begin{eqnarray}
v_j=v_i+2{\rm Re}[h_{ji}]-{\rm Re}[\bar\rho^i_a\alpha_{ji}+\rho^i_h\beta_{ji}],
\end{eqnarray}
where $v_j$ is the solution of (\ref{eq:0curv1}) on $U_j$. 
We now just double check that the jump from $v_i$ to $v_j$ is indeed a harmonic cocycle.
Indeed on $U_i\cup U_j\cup U_k$:
\[
v_k=v_j+2{\rm Re}[h_{kj}]-{\rm Re}[\bar\rho^j_a\alpha_{kj}+\rho^j_h\beta_{kj}]=
\]
\[ 
v_i+2{\rm Re}[h_{kj}] +2{\rm Re}[h_{ji}]-{\rm Re}[\bar\rho^i_a\alpha_{ji}+\rho^i_h\beta_{ji}]-{\rm Re}[\bar\rho^j_a\alpha_{kj}+\rho^j_h\beta_{kj}]=
\]
\[ 
v_i+2{\rm Re}[h_{kj}] +2{\rm Re}[h_{ji}]-{\rm Re}[\bar\rho^i_a\alpha_{ji}+\rho^i_h\beta_{ji}]-{\rm Re}[(\bar\rho^i_a+\beta_{ji})\alpha_{kj}+(\rho^i_h+\alpha_{ji})\beta_{kj}]=
\]
\[
v_i+2{\rm Re}[h_{ki}] -{\rm Re}[\alpha_{kj}\beta_{ji}-\alpha_{ji}\beta_{kj}]-{\rm Re}[\beta_{ji}\alpha_{kj}+\alpha_{ji}\beta_{kj}]
-{\rm Re}[\bar\rho^i_a\alpha_{kj}+\rho^i_h\beta_{kj}]=
\]
\[
v_i+2{\rm Re}[h_{ki}]-{\rm Re}[\bar\rho^i_a\alpha_{kj}+\rho^i_h\beta_{kj}],
\]
where in the third line we used the fact that $$\bar{\rho_j^a}=\bar{\rho^a_i}+\beta_{ji},\quad \rho_j^h=\rho^h_i+\alpha_{ji}$$ on $U_i\cap U_j$ and in the fourth line that 
$$
h_{ki}= h_{kj} + h_{ji} + \frac{1}{2} (\alpha_{kj} \beta_{ji} - \alpha_{ik} \beta_{kj}),
$$
and the cocycle condition for $\alpha, \beta$.
Therefore, indeed 
$$
t_{ji}=2{\rm Re}[h_{ji}]-{\rm Re}[\bar\rho^i_a\alpha_{ji}+\rho^i_h\beta_{ji}]
$$ 
is a harmonic cocycle as well as its holomorphic part:
$$
t^h_{ji}=h_{ji}-\frac{1}{2}(\bar{\rho^a_i}\alpha_{ji}+\rho^h_i\beta_{ji}).
$$

The real harmonic nature of the transformation for $v_j$ again guarantees the existence of such a harmonic solution, so that 
$\{e^{v_j}\}$ generates a harmonic metric on the holomorphic line bundle on $C$ specified by cocycle $\{t^h_{ji}\}$. 

The non-uniqueness of $H$ is hidden in constants of harmonic functions $\rho$ and 
$v$ which could shifted by a global $\SL(1|1)$ gauge transformation because of (\ref{eq: metrictransform}). 
It gives rise to a flat connection form, which is unique up to gauge transformation. Thus we arrive to the following theorem.

\begin{Thm}
The Hermitian metric $H$ solving the zero-curvature equations is unique up to global constant gauge transformations in $\SL(1|1)$. Consequently, the associated unitary flat connection is unique up to $\SL(1|1)$-gauge equivalence.
\end{Thm}

The following Corollary gives explicit solution for 1-cochain $\{f_{ij}\}$, such that $g=\delta(f)$ which we obtained in the process.
\begin{Cor}
The 2-cochain $t^h_{ji}=h_{ji}+f_{ji}$, where 
\begin{equation}
f_{ij}=\frac{1}{2}(\alpha_{ji}\bar{\rho^i_a}+\beta_{ji}\rho^i_h)
\end{equation} 
generates an element of $\check{H}^1(C, \mathcal{O})$, such that  $\delta(f)=g$ from (\ref{eq:2cocycle}), 
where $\rho^i=\rho^i_a+\rho^i_h$ corresponds to off-diagonal term of the Hermitian form solving zero curvature condition for the corresponding Chern connection.
\end{Cor}

Now let us proceed with the Hitchin equations, namely we modify zero curvature condition by:
\begin{eqnarray}\label{eq:Hitchin}
F=[\Phi, \Phi^{\dagger}_H], \quad \Phi^{\dagger}_H=H^{-1}\Phi^{\dagger}H
\end{eqnarray}
In our notation, locally we have:
\begin{eqnarray}
\Phi &=& \begin{pmatrix}
a & \delta \nonumber\\
\gamma & a 
\end{pmatrix}, 
\end{eqnarray}
The right hand side has the following form: 
\begin{eqnarray}
\Phi^{\dagger}_H=
\begin{pmatrix}
\star & \bar{\gamma} \nonumber\\
\bar{\delta} & \star
\end{pmatrix},
\quad 
[\Phi, \Phi^{\dagger}_H]=\begin{pmatrix} \delta\bar{\delta}+\gamma\bar{\gamma} & 0 \nonumber\\
0 & \delta\bar{\delta}+\gamma\bar{\gamma}
\end{pmatrix}.
\end{eqnarray}
Thus the component form for the Hitchin equations for $H$ are:
\begin{eqnarray}
&&\partial_{\bar z}\partial_z u=\delta\bar{\delta}+\gamma\bar{\gamma}+\frac{1}{2}(\partial_{\bar z}\bar \rho\partial_z\rho+\partial_{\bar z} \rho\partial_z\bar\rho),\label{eq:0hit1}\\
&&\partial_z \partial_{\bar z}\rho=0 \label{eq:0hit2}
\end{eqnarray}
and thus the local solutions are: 
\begin{equation}
\rho=\rho_h+\rho_a,\quad 
u=v+\frac{1}{2}\bar{\rho_h}{\rho}_h+\frac{1}{2}{\rho}_a\bar{\rho}_a+\eta\bar{\eta}+\phi\bar{\phi},
\end{equation}
where $v$ is harmonic $\rho_h$, $\rho_a$ are holomorphic and antiholomorphic parts of harmonic $\rho$, while $\eta=\int_z \delta$, $\phi=\int_z \gamma$.

To show the existence of the solution one should proceed the same as in the zero curvature case.

Thus  the corresponding flat connection 
\[
\nabla=\partial+\bar{\partial}+H^{-1}\partial H +\Phi+\Phi^{\dagger}
\]
is uniquely defined up to a gauge transformation and we can update above theorem above. 

\begin{Thm}
For a given $\SL(1|1)$ Higgs bundle $(E, \Phi)$, the Hermitian metric $H$ solving the Hitchin equations (\ref{eq:Hitchin}) is unique up to global constant gauge transformations in $\SL(1|1)$. Consequently, the associated flat connection is unique up to $\SL(1|1)$-gauge equivalence.
\end{Thm}

\section{Flat $\SL(1|1)$-connections and their parametrization}

The parametrization of flat $GL(1|1)$ connections using fatgraphs was done in \cite{bourque} in the real setting. Here we recall this construction. 
We consider Riemann surfaces ${F}$ with genus $g\geq 0$ and $s\geq 1$ punctures and ${C}$ stands for its compactification.

\begin{Def}
A \textit{fatgraph} (also known as \textit{ribbon graph}) is a graph with a cyclic ordering of the edges at each vertex. An \textit{orientation} on a fatgraph is an assignment of direction to each edge of the graph.
\end{Def}

One can reconstruct $F$ from a fatgraph $\tau$ by fattening the edges and vertices along with the cyclic ordering of edges at the vertices (see, e.g., \cite{pbook} for details). 
The condition $2g-2+s>0$ is necessary to choose a trivalent fatgraph for a given surface, as the number of vertices of such a graph is equal to $2(2g-2+s)$.
%There are Whitehead (flip) transformations as in Figure \ref{fig:graphflip}, which take the edge $e$ between 
%vertices $u, v$, adjacent to $c,d$ and $a, b$ respectively, shrinks it, and then extends an edge $f$ connecting vertices $u',v'$ adjacent to $b,c$ and $a,d$ respectively. Flips are known to act transitively on the collection of trivalent fatgraphs of $F$. Altogether they form a {\it Ptolemy groupoid ${\rm Pt}(F)$}. 
%Composition of flips are known to give generators for the mapping class group of a surface \cite{DTT}.

For a given fatgraph $\tau$ and a Lie (super)group $G$, we define a $G$-graph connection on $\tau$ as follows.

\begin{Def}
A $G$-\textit{graph connection} is the assignment to each edge $e$ of $\tau$ an element $g_{e}\in G$ and an orientation, so that $g_{\bar{e}}=g_e^{-1}$ if $\bar{e}$ is the inverse-oriented edge. We denote the space of $G$-graph connections as $\tilde{\rm M}_G(\tau)$. 
Two graph connections $\{g_e\}$, $\{\tilde{g}_e\}$ are \textit{gauge equivalent} if there is an assignment $v\rightarrow h_v\in G$, for all vertices $v\in \tau$, such that $\tilde{g}_e=h_v^{-1}g_eh_{v'}$, where $e$ starts and ends at $v'$ and $v$ respectively. We will denote the space of gauge equivalence classes of elements in $\tilde{\rm M}_G(\tau)$ by ${\rm M}_G(\tau)$. 
\end{Def}
\begin{Def}
 Let $\{B_i\}^s_{i=1}$ denote cycles surrounding punctures on the fatgraphs. Let us also denote the spaces $\tilde{\rm M}^c_G(\tau)$ and $\tilde{\rm M}^c_G(\tau)$ as graph connections and their equivalence classes with the condition that 
\begin{equation}\label{eq:monb}
\prod_{e\in B_i} g_e=1
\end{equation}
 for all $B_i$, where the order and orientation of edges follows their order and orientation of the oriented boundary. 
\end{Def}

%The space of graph connections modulo gauge 
%equivalences can be identified with a more common 
%differential-geometric object. 
%Namely, there is a natural one-to-one correspondence between ${\rm M}_G(\tau)$ and the moduli space of flat $G$-connections on ${F}$. This correspondence is constructed as follows: one identifies $M_A(e)$, 
%the monodromy of the flat connection $A$ along the oriented edge of the fatgraph $e$, with the group element $g_e$. This way, the space $\tilde{\rm M}_G(\tau)$ is identified with the space of flat connections modulo gauge transformations which are equal to identity at the vertices. Altogether, compositions of these group elements along the cycles of the fatgraph contain all information about the gauge classes of $A$. However, there is a residual gauge symmetry at the vertices of the fatgraph, which one has to take into account, and that is precisely the equivalence relation for graph connections. 
%Removing punctures, namely applying conditions  (\ref{eq:monb}) for boundary pieces, one obtains similar statement for $F$.

The moduli space of graph connections, up to gauge equivalences, admits a natural identification with a standard object from differential geometry. Specifically, there exists a canonical bijection between ${\rm M}_G(\tau)$ and the moduli space of flat $G$-connections over the surface $F$.
This bijection is defined by associating to each flat connection $A$ on $F$ its holonomies: 
the monodromy $M_A(e)$ around each oriented edge $e$ of the fatgraph is taken to be the corresponding group element 
$g_e$. In this manner, the space $\tilde{\rm M}_G(\tau)$ becomes identified with the space of flat connections modulo gauge transformations that fix the vertices. 
The holonomies along all cycles of the fatgraph then encode complete information about the gauge equivalence class of the connection $A$. The remaining gauge freedom -- consisting of independent gauge transformations at each vertex -- precisely corresponds to the gauge equivalence relation on graph connections.
When punctures are removed, by imposing the boundary monodromy conditions (\ref{eq:monb}) on the corresponding boundary segments, an analogous identification holds directly for the surface $C$.

One can formulate this as follows.
\begin{Thm}
1)If ${F}$ deformation retracts to $\tau$, then the moduli space of flat $G$-connections on ${F}$ is isomorphic to the space of gauge equivalent classes of $G$-graph connections on $\tau$ corresponding to $\tilde{F}$,  i.e. 
\begin{equation}
{\rm M}_G(\tau)\cong \Hom(\pi_1({F}), G)/G 
\end{equation}
2) If one applies boundary constraints (\ref{eq:monb}) , then we have isomorphism
\begin{equation}
{\rm M}^c_G(\tau)\cong \Hom(\pi_1({C}), G)/G 
\end{equation}
\end{Thm}

Let us fix a trivalent fatgraph $\tau\subset F$ with orientation $o$.
Assigning to every edge a group element of $\SL(1|1)$, or an ordered tuple $(h,\alpha,\beta)$ corresponding to the element $\tilde{g}(h,\alpha,\beta)$, we obtain a vector in the coordinate system $\tilde{C}(F,o,\tau)$ for the space of $\SL(1|1)$-graph connections without factoring by the gauge group equivalences.
%\begin{defff}
%We will refer to the space $\tilde{M}_{GL(1|1)}(F)$ of  $GL(1|1)$-graph connections without equivalence relations at vertices, as a decorated space of flat $GL(1|1)$-connections on $F$.
%\end{defff}
Thus the chart $\tilde{C}(F,o,\tau)$ gives a diffeomorphism:
\begin{eqnarray}
\tilde{\rm M}_{GL(1|1)}(\tau)\cong \mathbb{C}^{6g-6+3s|12g-12+6s}/ \mathbb{Z}^{6g-6+3s},
\end{eqnarray}
where the factor applied to the bosonic components.

Our goal is now to reformulate the gauge transformations at the vertices of a fatgraph as relations between coordinates in the charts $\tilde{C}(F,o,\tau)$. 

\begin{Def}
Let $\vec[c]\in \tilde C(F,o,\tau)$ be a coordinate vector. Suppose $e_1,e_2,e_3$ are edges oriented towards a vertex $u$ such that, for each $i=1,2,3$, $e_i$ has coordinates $(h_i;\alpha_i,\beta_i)$. Then a \textit{vertex rescaling} at $u$ is one of three coordinate changes for the coordinates of each $e_i$, where $c\in \Lambda_0$ and $\gamma\in \Lambda_1$:
\begin{itemize}
    \item $(h_i,\alpha_i,\beta_i)\mapsto (h_i+c,\alpha_i,\beta_i)$;
    \item $(h_i,\alpha_i,\beta_i)\mapsto (h_i+\gamma\alpha_i,\alpha_i+\gamma,\beta_i)$;
    \item $(h_i,\alpha_i,\beta_i)\mapsto (h_i+\gamma\beta_i,\alpha_i,\beta_i+\gamma).$
\end{itemize}

\end{Def}

It turns out that the edge reversal and the vertex rescalings define an equivalence relation on the chart $\tilde{C}(F, o, \tau)$. The vertex rescalings come from the equivalences on $GL(1|1)$ graph connections provided by the appropriate 1-parameter subgroups. This leads to the following theorem.

\begin{Thm}
Let $C(F,\tau) = \tilde C(F,o,\tau)/\sim$ be the quotient by the equivalences provided by edge reversal and vertex rescalings. Then $C(F,\tau)$ is in bijection with the moduli space ${\rm M}_{GL(1|1)}(\tau)$.
\end{Thm}

\begin{proof}
Indeed, all those 3 types of rescalings correspond to gauge transformations at the vertices for the 1-parameter subgroups corresponding to upper/lower triangular and  diagonal matrices.
\end{proof}

Now one can simply constrain the vertex rescalings. To do that let us introduce the following notation: for a given edge with assignment $(h,\alpha, \beta)$, which is adjacent to vertex $v$, we denote $h^v=h$, $b^v=b$, $\alpha^v=\alpha$, $\beta^v=\beta$, if edge is oriented towards $v$, and $h^v=-h$, $\alpha^v=-\alpha$, $\beta^v=-\beta$ otherwise.
Then, let us give the following definition.
\begin{Def}
Let $\vec[c]\in\tilde{C}(F,o,\tau)$ so that $\{(h_i, \alpha_i, \beta_i)\}_{i=1,2,3}$ are the assignments for edges adjacent to vertex $v$. We refer to the following conditions as \textit{gauge constraints} at vertex $v$:
\begin{eqnarray}
\sum_{i=1}^3 h_i=\sum_{i=1}^3{\alpha^v_i}=\sum_{i=1}^3{\beta^v_i}=0.
\end{eqnarray}
\end{Def}

One can see that, for fixed orientation $o$, such constraints pick exactly one element from the equivalence classes provided by vertex rescalings. Moreover, we have the following theorem.

\begin{Thm}
The coordinate chart $\tilde{C}(F,o,\tau)$ modulo gauge constraints at all vertices $v$ of $\tau$ provides a complex-analytic isomorphism
\begin{eqnarray}
{\rm M}_{\SL(1|1)}(\tau)\cong\mathbb{C}_{\Lambda}^{2g+2s-1|4g+2s-2}/\mathbb{Z}^{2g+2s-1},
\end{eqnarray}
and applying the $s$ constraints at the punctures, we obtain another complex-analytic isomorphism:
\begin{eqnarray}
{\rm M}^c_{\SL(1|1)}(\tau)\cong\mathbb{C}_{\Lambda}^{2g|4g}/\mathbb{Z}^{2g}
\end{eqnarray}
\end{Thm}

One can further restrain this construction, by picking a real form for $g(h, \alpha, \beta)$, such that $\bar{h}=-h$ and $\bar{\alpha}=-\beta$. That corresponds to the supergroup $\rm SU(1|1)$, and thus we obtain that 
\begin{eqnarray}
{\rm M}_{\rm SU(1|1)}(\tau)\cong\mathbb{R}_{\Lambda^{\mathbb{R}}}^{2g+2s-1|2g+s-1}/\mathbb{Z}^{2g+2s-1}; \quad {\rm M}^c_{\rm SU(1|1)}(\tau)\cong\mathbb{R}_{\Lambda^{\mathbb{R}}}^{2g|2g}/\mathbb{Z}^{2g}
\end{eqnarray}
Recall that in section 7 
we constructed Chern connections based on our parametrization of bundles and the corresponding Hermitian form. We note that in the presence of Higgs field there are certain bilinear constraints on the parameters, which have to be taken into account.   
Thus we obtain the following Theorem, which 
serves as the analogue of Narasimhan-Seshadri theorem and nonabelian Hodge correspondence for $\SL(1|1)$.
\begin{Thm}
1)There is a 1-to-1 correspondence between the space of holomorphic $\SL(1|1)$ bundles of degree 0 and ${\rm M}^c_{\rm SU(1|1)}$.\\ 
2) There is an injective map from  the moduli space $\SL(1|1)$ Higgs bundles of degree 0 to ${\rm M}^c_{\rm SL(1|1)}$.
\end{Thm}
\begin{proof}
Indeed, the harmonic metric and related Chern connection provide the explicit injective map. However, given the constraints on the odd parts of Higgs field the map from  $\SL(1|1)$ Higgs bundles to ${\rm M}^c_{\rm SL(1|1)}$ is not surjective.
\end{proof}

\section {$\GL(1|1)$ Hitchin system on $\mathbb{P}^1$, Garnier and Gaudin models}

\subsection{Parabolic structures and Garnier system} 
In this section we will look at the Hitchin system associated with $\GL(1|1)$ Higgs bundles with parabolic structures. 

\begin{Def} A parabolic Higgs superbundle on $\mathbb{P}^1$ for $GL(1|1)$ is a triple $(E, \Phi, \{F_i, a_i, b_i\})$, where:
\begin{itemize}
    \item $E \to \mathbb{P}^1$ is a rank-$(1|1)$ holomorphic superbundle.
    \item $\Phi: E \to E \otimes K_{\mathbb{P}^1}(D)$ is a Higgs field, where $D = \{z_1, \ldots, z_m\}$ is a divisor of $m$ marked points, and $K_{\mathbb{P}^1}(D) = \mathcal{O}(-2 + m)$.
    \item At each $z_i$, $F_i \subset E_{z_i} \cong \Lambda^{1|1}$ is a $(1|0)$-dimensional even subspace, so that the residue 
$\text{Res}_{z=z_i}(\Phi) = A_i \in \mathfrak{gl}(1|1)$ 
preserves $F_i$, so that $a_i\in \Lambda_0$ is the corresponding eigenvalue, and $\str(A_i)=a_i-b_i$.
\end{itemize}
\end{Def}
\noindent Assuming that we have at most a residue at infinity, the Higgs field is:
\[
\Phi(z) = \sum_{i=1}^m \frac{A_i}{z - z_i} \, dz, \quad A_i \in \mathfrak{gl}(1|1).
\]
%where 
%\[
%\sum^m_{i=1}A_i=0.
%\] 
In a local basis at $z_i$ where $F_i = \text{span} \begin{pmatrix} 1 \\ 0 \end{pmatrix}$, the residue is:

\[
A'_i = \begin{pmatrix} a_i & \theta_i \\ 0 & b_i \end{pmatrix}, \quad \theta_i \in \Lambda_1.
\]

\noindent Assuming the trivialization, e.g. $E \cong \mathcal{O} \oplus \Pi \mathcal{O}$ on $\mathbb{P}^1$, the flag at the given point $p_i$ is:

\[
F_i = \text{span} \begin{pmatrix} 1 \\ \eta_i \end{pmatrix}, \quad \eta_i \in \Lambda_1.
\]

\noindent The gauge transformation to transform the flag is:

\[
A'_i\to A_i=g_i A'_i g^{-1}_i, \quad 
g_i = \begin{pmatrix} 1 & 0 \\ \eta_i & 1 \end{pmatrix}, \quad g_i^{-1} = \begin{pmatrix} 1 & 0 \\ -\eta_i & 1 \end{pmatrix},
\]
so that:
\[
A_i = \begin{pmatrix}a_i - \theta_i \eta_i & \theta_i \\ (a_i-b_i) \eta_i & b_i - \theta_i \eta_i \end{pmatrix}.
\]
The expression for the super-symplectic form on the space of such Higgs bundles is given by:
\[
\omega =  \sum_{i=1}^m \text{Res}_{z_i} \text{str}(\delta E \wedge \delta \Phi).
\]
To compute the contributions to residue term we have the following expression as deformation of the fiber at the parabolic point:
%\[
%\delta A_i = \begin{pmatrix} \delta a_i & \delta \theta_i \\ 0 & \delta b_i \end{pmatrix}, \quad 
\[\delta E_i = \begin{pmatrix} \delta e_{i,11} & \delta e_{i,12} \\ \delta e_{i,21} & \delta e_{i,22} \end{pmatrix},
\]
and the flag variation gives:
\[
\delta \begin{pmatrix} 1 \\ \eta_i \end{pmatrix} = \begin{pmatrix} 0 \\ \delta \eta_i \end{pmatrix} .
\]
Therefore the only nonzero term is $\delta e_{i,21}= \delta \eta_i$, 
and thus $
\text{Res}_{z_i} \text{str}(\delta E \wedge \delta \Phi) =  \delta \eta_i \wedge \delta \theta_i$ and the symplectic form is simply given by:
\[
\omega =  \sum_{i=1}^m \delta \eta_i \wedge \delta \theta_i.
\]
The supertrace gives rise to commutative family of Hamiltonians.
Indeed, 
\[
\text{str}(\Phi^2) = \sum_{i,j=1}^m \frac{\text{str}(A_i A_j)}{(z - z_i)(z - z_j)} \, dz^2.
\]
Using the fact that 
\[
\frac{1}{(z-z_i)(z-z_j)}=
\frac{1}{z_i-z_j}\Big(\frac{1}{z-z_i}-\frac{1}{z-z_j}\Big)
\]
we compute the Hamiltonians as residues at $z_i$:
\[
H_i = \sum_{j \neq i} \frac{\text{str}(A_i A_j)}{z_i - z_j},
\]
Now rearranging the terms in $\text{str}(A_i A_j)$ we obtain:
\[
\text{str}(A_i A_j) = a_i a_j - b_i b_j -(a_i-b_i) \eta_i\theta_j + b_j \theta_i \eta_i- a_j \theta_i\eta_i - a_i \theta_j \eta_j + b_i \theta_j \eta_j + (a_j-b_j ) \theta_i \eta_j=
\]
%\[
%\frac{1}{2}(a_i +b_i) (a_j-b_j)+\frac{1}{2}(a_i-b_i)(a_j+b_j) - (a_j-b_j) \theta_i\eta_i  - (a_i-b_i) \theta_j \eta_j + (a_j-b_j ) \theta_i \eta_j-(a_i-b_i) \eta_i\theta_j=
%\]
%\[
%\frac{1}{2}(a_i +b_i) (a_j-b_j)+\frac{1}{2}(a_i-b_i)(a_j+b_j) - (a_j-b_j)\theta_i\eta_i  - (a_i-b_i) \theta_j\eta_j  + (a_j-b_j ) \theta_i \eta_j-(a_i-b_i) \eta_i\theta_j=
%\]
\[
\frac{1}{2}(a_i +b_i-2\theta_i\eta_i) (a_j-b_j)+\frac{1}{2}(a_i-b_i)(a_j+b_j-2\theta_j\eta_j)+ (a_j-b_j ) \theta_i \eta_j-(a_i-b_i) \eta_i\theta_j.
\]
Denoting $u_i=a_i+b_i$, $v_i=a_i-b_i$, we have:
\[
H_i = \sum_{j \neq i} \frac{\frac{1}{2}(u_i-2\theta_i\eta_i) v_j+\frac{1}{2}v_i(u_j-2\theta_j\eta_j)+  \theta_i v_j\eta_j-v_i \eta_i\theta_j}{z_i - z_j}.
\]
This family of commuting Hamiltonians gives the $\mathfrak{gl}(1|1)$ superanalogue of the twisted Garnier system. 

\subsection{Quantization and Gaudin model}

The quantization of $\omega$ assigns
$$
\eta_i\to \hbar \partial_{\theta_i}, \quad u_i\to \hbar {\rm u}_i, \quad v_i\to {\rm v}_i,
$$
and we obtain:
\begin{equation}
H_i = \hbar\sum_{j \neq i} \frac{\frac{1}{2}({\rm u}_i-2\theta_i\partial_i) {\rm v}_j+\frac{1}{2}{\rm v}_i({\rm u}_j-2\theta_j\partial_j)+  \theta_i {\rm v}_j\partial_{\theta_j}-{\rm v}_i \partial_{\theta_i}\theta_j}{z_i - z_j}.
\end{equation}
Denoting 
$$
N_i:=\frac{1}{2}{\rm u}_i-\theta_i\partial_{\theta_i},\quad  E_i:={\rm v}_i, \quad \Psi_i^{-}=\theta_i, \quad \Psi_i^{+}={\rm v}_i\partial_{\theta_i}
$$
we obtain that $N_i, E_i, \Psi^{\pm}_i$ are the generators of two-dimensional $\mathfrak{gl}(1|1)$ modules $\mathbb{C}[\theta_i]$, so that commutations relations are: 
$$
[N_i, \Psi^{\pm}_i]=\pm \Psi^{\pm}_i, \quad [\Psi^{+}_i, \Psi^{-}_i]=E_i, \quad [E_i, \Psi^{\pm}_i]=[E_i, N_i]=0. 
$$
Therefore, we have 
\[
H_i =\hbar\sum_{j \neq i}\frac{E_iN_j+N_iE_j+\Psi^{-}_i\Psi^+_j-\Psi^{+}_i\Psi^-_j}{z_i - z_j}.
\]
acting in 
\[
{\mathrm V}({\rm u}_1, {\rm v}_1)\otimes {\mathrm V}({\rm u}_2, {\rm v}_2)\otimes \dots \otimes {\mathrm V}({\rm u}_m, {\rm v}_m)
\]
One can immediately recognize $\{H_i\}^m_{i=1}$ as Hamiltonians of $\mathfrak{gl}(1|1)$ Gaudin model, acting in the tensor product of two-dimensional modules ${\mathrm V}({\rm u}_i,{\rm v}_i)$ isomorphic to $\mathbb{C}[\theta_i]$, as vector spaces,and  parametrized by continuous parameters ${\rm u}_i,{\rm v}_i$. One can find the full description of this model and Bethe ansatz solution in \cite{KangLu_Gaudin}.

\bibliography{biblio}
\end{document}